\documentclass[12pt]{amsart}

\usepackage{amsmath,amssymb,amsthm,mathtools}
\usepackage[a4paper,margin=1in]{geometry}
\usepackage{hyperref}
\usepackage{enumitem}

\newtheorem{theorem}{Theorem}[section]
\newtheorem{lemma}[theorem]{Lemma}
\newtheorem{proposition}[theorem]{Proposition}
\newtheorem{corollary}[theorem]{Corollary}
\newtheorem{remark}[theorem]{Remark}
\theoremstyle{definition}

\newtheorem{conjecture}[theorem]{Conjecture}
\newtheorem{example}[theorem]{Example}

\DeclareMathOperator{\rad}{rad}

\title{Valuation Separation for Coprime Lucas Products}
\author{Dongyeon Kym}
\email{kynteger@dgu.ac.kr}
\date{}



\begin{document}
	
	\maketitle
	
\begin{abstract}
	Let \(U_n=U_n(P,Q)\) be a nondegenerate Lucas sequence with \(Q=\pm1\) and positive discriminant \(\Delta=P^2+4Q\).
	We study equations
	\[
	A y^k=\prod_{i=1}^r U_{n_i}(P,Q),\qquad k\ge 2,
	\]
	where the indices \(n_1,\ldots,n_r\) are pairwise coprime.
	The strong divisibility property makes the factors \(U_{n_i}\) pairwise coprime, so every local valuation defect outside \(A\) is forced to occur in a single factor.
	For \(k=2\), this gives a termwise square-class restriction: each \(U_{n_i}\) has signed squarefree part supported on the primes dividing \(A\).
	In particular, \(\Delta y^2=U_mU_n\), \((m,n)=1\), reduces to a finite square-class compatibility condition together with an integrality condition.
	Assuming the number-field \(abc\) conjecture over \(\mathbb Q(\sqrt\Delta)\), we prove finiteness for Lucas terms with squarefree part supported on a fixed finite set, and hence reduce the coprime product equations to a finite compatibility problem.
	We also give the \(k\)-th power analogue and a primitive-divisor obstruction.
\end{abstract}

\section{Introduction}

Let \(P,Q\in \mathbb Z\), and let \(U_n=U_n(P,Q)\) be the Lucas sequence of the first kind defined by
\[
U_0=0,\qquad U_1=1,\qquad U_{n+2}=P U_{n+1}+Q U_n .
\]
Throughout the paper we assume that
\[
Q=\pm 1,\qquad \Delta=P^2+4Q>0,
\]
and that the sequence is nondegenerate.
This sign convention differs from the common \(U_n(P,Q)\) convention by replacing \(Q\) with \(-Q\).
If \(\alpha,\beta\) are the roots of \(x^2-Px-Q\), then
\[
\alpha+\beta=P,\qquad \alpha\beta=-Q=\pm1,\qquad
\alpha-\beta=\sqrt\Delta,
\]
and
\[
U_n=\frac{\alpha^n-\beta^n}{\alpha-\beta}.
\]
Nondegeneracy means that \(\alpha/\beta\) is not a root of unity.

The purpose of this paper is to isolate and exploit a valuation-separation principle for coprime products of Lucas sequence terms. We study equations of the form
\begin{equation*}\label{eq:product}
A y^k=\prod_{i=1}^r U_{n_i}(P,Q), \qquad k\ge 2,
\end{equation*}
where \(A\in \mathbb Z\setminus\{0\}\) is fixed and the indices \(n_1,\ldots,n_r\) are pairwise coprime.
The principal case is \(k=2\), where the equation becomes a square-class problem; the same local separation mechanism also applies to higher powers.

The point of the coprimality hypothesis is not that the indices are arithmetically simple, but that it prevents local valuation defects from being shared among different Lucas factors.
Indeed, for \(Q=\pm1\) the sequence \(U_n(P,Q)\) is a strong divisibility sequence:
\[
\gcd(|U_a|,|U_b|)=|U_{\gcd(a,b)}|.
\]
Thus pairwise coprime indices give pairwise coprime Lucas terms.
If a prime \(p\nmid A\) divides one of the factors \(U_{n_i}\), then it divides no other factor.
Hence a global \(k\)-th power condition on the product forces
\[
v_p(U_{n_i})\equiv 0 \pmod{k}
\]
in the individual factor itself.
This is the basic valuation-separation principle of the paper.

The first result is the square-class form of the valuation-separation principle.

\begin{theorem}\label{thm:multi-rigidity}
	Let \(A\in\mathbb Z\setminus\{0\}\), and suppose that \(n_1,\ldots,n_r\) are pairwise coprime positive integers. If
	\[
	A y^2=\prod_{i=1}^r U_{n_i}(P,Q)
	\]
	for some \(y\in\mathbb Z\), then for every \(i\) there exist a signed squarefree integer \(e_i\) and an integer \(s_i\) such that
	\[
	U_{n_i}=e_i s_i^2,\qquad |e_i|\mid \operatorname{rad}(A).
	\]
Equivalently, for every \(i\) one has
\[
[U_{n_i}(P,Q)]\in G_A,
\]
where
\[
G_A=\langle -1,p:p\mid A\rangle\subset \mathbf Q^\times/(\mathbf Q^\times)^2.
\]
\end{theorem}

The theorem should be viewed as a product-to-factor principle.
It does not assert that a Lucas term has small squarefree part in isolation.
Rather, it says that if such a term occurs inside a coprime product which is a square up to the fixed coefficient \(A\), then all odd valuation defects outside \(A\) are forbidden term by term.

In the two-term equation
\[
\Delta y^2=U_m(P,Q)U_n(P,Q),\qquad \gcd(m,n)=1,
\]
this gives a finite square-class criterion. For an integer \(A\ne0\), put
\[
G_A=\langle -1,p:p\mid A\rangle
\subseteq \mathbb Q^\times/(\mathbb Q^\times)^2.
\]
The preceding theorem implies that every integral solution satisfies
\[
[U_m],[U_n]\in G_\Delta.
\]
On the other hand, the equation itself is equivalent over \(\mathbb Q\) to
\[
[U_m][U_n]=[\Delta]
\qquad\text{in }\mathbb Q^\times/(\mathbb Q^\times)^2.
\]
Thus the possible square classes are governed by a finite compatibility relation inside \(G_\Delta\).
The distinction between rational and integral solutions is then exactly the additional integrality condition that \(U_mU_n/\Delta\) be an integer square.

The second part of the paper uses the number-field \(abc\) conjecture to turn this termwise square-class restriction into an index finiteness statement.
If \(S\) is a finite set of rational primes, define
\[
G_S=\langle -1,p:p\in S\rangle
\subseteq \mathbb Q^\times/(\mathbb Q^\times)^2.
\]
We prove the following conditional theorem.

\begin{theorem}\label{thm:conditional}
	Assume the \(abc\) conjecture over \(K=\mathbb Q(\sqrt\Delta)\).
	Let \(S\) be a finite set of rational primes. Then the set of positive integers \(n\) for which
	\[
	[U_n(P,Q)]\in G_S
	\]
	is finite.
	Equivalently, there are only finitely many \(n\) for which
	\[
	U_n=e s^2,
	\]
	where \(e\in\mathbb Z\) is signed squarefree and
	\(\operatorname{Supp}(e)\subseteq S\).
\end{theorem}

Combining the unconditional rigidity theorem with this conditional \(S\)-square finiteness theorem gives a uniform bound for all indices appearing in coprime product equations.

\begin{corollary}\label{cor:indices-bound}
	Assume the \(abc\) conjecture over \(K=\mathbb Q(\sqrt\Delta)\).
	Fix \(A\in\mathbb Z\setminus\{0\}\) and \(r\ge1\).
	Then there exists a constant \(C=C(P,Q,A)\) such that, whenever \(n_1,\ldots,n_r\) are pairwise coprime and
	\[
	A y^2=\prod_{i=1}^r U_{n_i}(P,Q)
	\]
	for some \(y\in\mathbb Z\), one has
	\[
	\max_i n_i<C.
	\]
	In particular, for fixed \(A\) and \(r\), only finitely many pairwise coprime index \(r\)-tuples can occur.
\end{corollary}

\begin{corollary}[Finite reduction under \(abc\)]\label{thm:finite-reduction}
Assume the number-field \(abc\) conjecture over \(K=\mathbb Q(\sqrt{\Delta})\).
Fix \(A\in\mathbb Z\setminus\{0\}\), and put
\[
T(P,Q,A)=\{n\geq 2:[U_n(P,Q)]\in G_A\}.
\]
Then \(T(P,Q,A)\) is finite. Moreover, after deleting the trivial factors \(U_1=1\), every
solution of
\[
Ay^2=\prod_{i=1}^r U_{n_i}(P,Q),\qquad (n_i,n_j)=1\ (i\neq j),
\]
has all its indices in \(T(P,Q,A)\). Define the coprimality graph on \(T(P,Q,A)\) to be the graph whose vertices are the elements of \(T(P,Q,A)\) and whose edges join pairs of coprime integers.
Then the nontrivial indices form a clique in this graph.
The remaining conditions are the finite square-class compatibility condition
\[
\prod_i [U_{n_i}]=[A]\quad\text{in }
\mathbb Q^\times/(\mathbb Q^\times)^2
\]
together with the corresponding integrality condition.
\end{corollary}

We also obtain an unconditional primitive-divisor obstruction.
If a large index \(n\) occurs in a coprime product solution and \(n\) is not the rank of apparition of any prime divisor of \(A\), then every primitive prime divisor \(p\nmid A\) of \(U_n\) must satisfy
\[
v_p(U_n)\geq 2.
\]
Thus large admissible indices force a Lucas--Wieferich type phenomenon at their primitive divisors outside the fixed coefficient.

The valuation-separation argument is not intrinsically quadratic.
For this reason we formulate the \(k\)-th power analogue in Section~7. If
\[
A y^k=\prod_i U_{n_i}(P,Q),\qquad k\geq 2,
\]
with pairwise coprime indices, then every prime \(p\nmid A\) occurs in each individual factor \(U_{n_i}\) with valuation divisible by \(k\).
Thus each \(U_{n_i}\) is a \(k\)-th power up to primes dividing \(A\).
Under the same number-field \(abc\) hypothesis, the corresponding \(S\)-supported \(k\)-th power condition again gives a uniform bound for all occurring indices.

\paragraph{Relation with previous work.}
Primitive divisor theorems for Lucas and Lehmer sequences, beginning with Zsigmondy's theorem~\cite{Zsigmondy} and Carmichael's work~\cite{Carmichael1913} and culminating in the theorem of Bilu--Hanrot--Voutier~\cite{BiluHanrotVoutier2001}, show that sufficiently large Lucas terms possess new prime divisors.
Perfect powers and \(S\)-unit multiples of powers in linear recurrence sequences have also been studied extensively, for instance by Pethő~\cite{Petho1982} and by Bugeaud--Mignotte--Siksek~\cite{BugeaudMignotteSiksek2006}.
Products of Lucas sequence terms were studied by Luca and Shorey~\cite{LucaShorey2009} in the setting of perfect powers with indices lying in intervals.

The present paper uses these classical inputs in a somewhat different direction.
We do not prove a new primitive divisor theorem, nor do we attempt to classify perfect powers among individual recurrence terms.
Rather, we consider product equations in which the indices are required to be pairwise coprime. 
In this setting the strong divisibility property has a direct consequence: the corresponding Lucas terms are pairwise coprime, and hence a local valuation defect outside the fixed coefficient cannot be shared between different factors.
Thus a global power condition on the product imposes a termwise square-class, or more generally power-class, condition on the individual factors.

The role of the \(abc\) conjecture is then separated from this elementary valuation argument.
Once valuation separation has reduced the product equation to an \(S\)-supported square-class or power-class condition for single Lucas terms, the number-field \(abc\) conjecture over \(\mathbb Q(\sqrt{\Delta})\) is used to obtain finiteness for those terms.
This yields, under \(abc\), a finite reduction of the original coprime product equations.

The paper is organised as follows.
Section~2 recalls the elementary Lucas sequence facts used throughout the paper, especially the strong divisibility
property.
Section~3 proves the unconditional valuation-parity rigidity theorem for coprime products.
Section~4 records the square-class criteria for rational and integral solutions.
Section~5 gives the \(abc\)-conditional finiteness theorem for Lucas terms with squarefree part supported on a fixed finite set
of primes.
Section~6 gives the finite reduction of coprime product equations and records the primitive-divisor obstruction.
Section~7 develops the power-class versions of valuation separation.
Section~8 gives examples.

	\section{Lucas sequence preliminaries}
	
	Let $\alpha,\beta$ be the roots of $x^2-Px-Q$.  Since $Q=\pm1$, we have
	$|\alpha\beta|=1$. Since the sequence is nondegenerate, \(U_n\neq 0\) for every \(n\ge 1\); indeed \(U_n=0\) would imply \((\alpha/\beta)^n=1\). Hence in every equation considered below the variable \(y\) is automatically nonzero.
	
	\begin{lemma}\label{lem:growth}
		After relabelling $\alpha$ and $\beta$ if necessary, one has
		\[
		|\alpha|>1>|\beta|.
		\]
		Consequently,
		\[
		\log |U_n|=n\log |\alpha|+O_{P,Q}(1)
		\qquad(n\ge1).
		\]
	\end{lemma}
	
	\begin{proof}
		Since $|\alpha\beta|=1$, either both roots have absolute value $1$ or one has absolute value greater than $1$ and the other has absolute value less than $1$.
		If both had absolute value $1$, then both algebraic integers would have all conjugates on the unit circle.
		By Kronecker's theorem they would be roots of unity, and hence $\alpha/\beta$ would be a root of unity, contradicting
		nondegeneracy. Relabel so that $|\alpha|>1>|\beta|$.
		
		Binet's formula gives
		\[
		U_n=\frac{\alpha^n-\beta^n}{\alpha-\beta}
		=\frac{\alpha^n}{\sqrt\Delta}
		\left(1-\left(\frac{\beta}{\alpha}\right)^n\right).
		\]
		Since $|\beta/\alpha|<1$, the factor in parentheses is bounded above and below away from zero for all sufficiently large $n$.  The finitely many remaining values of $n$ are absorbed into the error term.
	\end{proof}

	\begin{lemma}\label{lem:strong-divisibility}
		For all positive integers \(a,b\), one has
		\[
		\gcd(|U_a|,|U_b|)=|U_{\gcd(a,b)}|.
		\]
		In particular, if \(\gcd(a,b)=1\), then \(\gcd(U_a,U_b)=1\).
	\end{lemma}
	
	\begin{proof}
		We recall the standard argument.
		Since \(Q=\pm1\), consecutive terms of the sequence are coprime.
		Indeed, any common divisor of \(U_n\) and \(U_{n+1}\) also divides \(Q U_{n-1}=U_{n+1}-P U_n\), and induction gives
		\(\gcd(U_n,U_{n+1})=1\).
		
		The Lucas addition formula
		\[
		U_{m+n}=U_mU_{n+1}+Q U_{m-1}U_n
		\]
		implies the Euclidean reduction
		\[
		\gcd(U_{m+n},U_n)=\gcd(U_m,U_n),
		\]
		because \(\gcd(U_n,Q U_{n+1})=1\).
		Iterating this reduction along the Euclidean algorithm gives
		\[
		\gcd(|U_a|,|U_b|)=|U_{\gcd(a,b)}|.
		\]
		The final assertion follows from \(U_1=1\).
	\end{proof}
	This strong divisibility property is classical; see, for example, Everest, van der Poorten, Shparlinski and Ward~\cite[Chapter 3]{EverestPoortenShparlinskiWard2003}.

	\section{Valuation rigidity for coprime products}

We now prove the unconditional valuation-parity theorem for pairwise coprime products.

\begin{lemma}\label{lem:pairwise-coprime-terms}
	If $n_1,\ldots,n_r$ are pairwise coprime positive integers, then the nonzero
	integers $U_{n_1},\ldots,U_{n_r}$ are pairwise coprime.
\end{lemma}

\begin{proof}
	For $i\ne j$, Lemma \ref{lem:strong-divisibility} gives
	\[
	\gcd(|U_{n_i}|,|U_{n_j}|)=|U_{\gcd(n_i,n_j)}|=|U_1|=1.
	\]
\end{proof}

\begin{lemma}\label{lem:valuation-parity-multi}
	Let $A\in\mathbb Z\setminus\{0\}$, and suppose that
	$n_1,\ldots,n_r$ are pairwise coprime.  If
	\[
	A y^2=\prod_{i=1}^r U_{n_i}
	\]
	with $y\in\mathbb Z$, then for every rational prime $p\nmid A$ and every
	$1\le i\le r$ one has
	\[
	v_p(U_{n_i})\equiv0\pmod2.
	\]
\end{lemma}

\begin{proof}
	Fix a rational prime $p\nmid A$.  Taking $p$-adic valuations in the equation
	gives
	\[
	2v_p(y)=\sum_{i=1}^r v_p(U_{n_i}).
	\]
	The right-hand side is even.  By Lemma \ref{lem:pairwise-coprime-terms}, the integers $U_{n_i}$ are pairwise coprime.  Hence $p$ divides at most one of them. Therefore each individual valuation $v_p(U_{n_i})$ is even.
\end{proof}

\begin{proof}[Proof of Theorem \ref{thm:multi-rigidity}]
	For each $i$, write
	\[
	U_{n_i}=e_i s_i^2,
	\]
	where $e_i=\operatorname{sqf}(U_{n_i})$ is signed squarefree and $s_i\in\mathbb Z$.  If a rational prime $p$ divides $e_i$, then $v_p(U_{n_i})$ is odd. By Lemma \ref{lem:valuation-parity-multi}, this is impossible unless $p\mid A$.  Hence every prime divisor of $e_i$ divides $A$, and therefore $|e_i|\mid\rad(A)$.
\end{proof}

\begin{corollary}\label{cor:two-term-rigidity}
	Suppose that
	\[
	\Delta y^2=U_mU_n,
	\qquad
	\gcd(m,n)=1.
	\]
	Then there exist signed squarefree integers $e_m,e_n$ and integers $s_m,s_n$
	such that
	\[
	U_m=e_m s_m^2,
	\qquad
	U_n=e_n s_n^2,
	\qquad
	|e_m|,|e_n|\mid\rad(\Delta).
	\]
\end{corollary}

\begin{remark}
	No primality assumption on $m$ or $n$ is used.
	The decisive hypothesis is the coprimality of the indices, which makes the corresponding Lucas terms coprime.
	This is why the result is naturally a theorem about coprime products rather than about prime indices.
\end{remark}

	\section{Square classes and exact criteria}
	
	The valuation-separation theorem gives necessary termwise restrictions.
	In order to describe the remaining finite problem, we also need the exact global square-class compatibility condition.
	This section records this condition separately, distinguishing the rational square-class obstruction from the additional integrality requirement.
	
	For a nonzero integer $N$, write
	\[
	N=\operatorname{sqf}(N)\,t^2,
	\]
	where $\operatorname{sqf}(N)$ is the signed squarefree part of $N$. Thus $\operatorname{sqf}(N)$ is squarefree up to sign and records the class of $N$ in $\mathbb Q^\times/(\mathbb Q^\times)^2$.
	
	For a finite set $S$ of rational primes, put
	\[
	\mathcal G_S
	:=\langle -1,\ p:p\in S\rangle
	\subset \mathbb Q^\times/(\mathbb Q^\times)^2.
	\]
	If $A\ne0$ is an integer, write
	\[
	\mathcal G_A:=\mathcal G_{\operatorname{Supp}(A)}.
	\]
	Equivalently,
	\[
	\mathcal G_A=\langle -1,\ p:p\mid A\rangle.
	\]
	Here and throughout, the notation \(p\mid A\) means \(p\mid |A|\).
	
	\begin{proposition}\label{prop:two-term-criterion}
		Let $m,n$ be positive integers.  The equation
		\[
		\Delta y^2=U_mU_n
		\]
		has a solution $y\in\mathbb Q^\times$ if and only if
		\[
		[U_m][U_n]=[\Delta]
		\qquad\text{in }\mathbb Q^\times/(\mathbb Q^\times)^2.
		\]
		It has a solution $y\in\mathbb Z$ if and only if, in addition,
		\[
		\frac{U_mU_n}{\Delta}\in \mathbb Z^2.
		\]
		Equivalently, $[U_m][U_n]=[\Delta]$ and the rational square $U_mU_n/\Delta$ is an integer square.
	\end{proposition}
	
	\begin{proof}
		The equation is equivalent to
		\[
		y^2=\frac{U_mU_n}{\Delta}.
		\]
		It has a nonzero rational solution if and only if $U_mU_n/\Delta$ is a square
		in $\mathbb Q^\times$, which is exactly
		\[
		[U_m][U_n]=[\Delta].
		\]
		The solution is integral precisely when the rational square $U_mU_n/\Delta$ is an integer square.
	\end{proof}
	
	\begin{corollary}\label{cor:finite-graph}
	Suppose that \(\gcd(m,n)=1\) and
	\[
	\Delta y^2=U_mU_n
	\]
	for some \(y\in\mathbb Z\). Then
	\[
	[U_m],[U_n]\in G_\Delta,
	\qquad
	[U_m][U_n]=[\Delta].
	\]
	Thus the square classes of the two terms lie in the finite group \(G_\Delta\) and satisfy the finite compatibility condition determined by \([\Delta]\).
	\end{corollary}

	\begin{proof}
		By Theorem~\ref{thm:multi-rigidity} with \(A=\Delta\) and \(r=2\), the coprimality
		assumption \(\gcd(m,n)=1\) implies
		\[
		[U_m],[U_n]\in G_\Delta.
		\]
		The identity
		\[
		[U_m][U_n]=[\Delta]
		\]
		follows from Proposition~\ref{prop:two-term-criterion}. This proves the claim.
	\end{proof}

	\begin{proposition}\label{prop:product-square}
		Let \(A\in \mathbb Z\setminus\{0\}\), and let
		\(n_1,\ldots,n_r\) be positive integers. Put
		\[
		B=\prod_{i=1}^r U_{n_i}(P,Q).
		\]
		Then the equation
		\[
		A y^2=B
		\]
		has a solution \(y\in\mathbb Q^\times\) if and only if
		\[
		\prod_{i=1}^r [U_{n_i}]=[A]
		\quad\text{in } \mathbb Q^\times/(\mathbb Q^\times)^2.
		\]
		It has a solution \(y\in\mathbb Z\) if and only if
		\[
		\prod_{i=1}^r [U_{n_i}]=[A]
		\]
		and, for every rational prime \(p\),
		\[
		\sum_{i=1}^r v_p(U_{n_i})-v_p(A)\in 2\mathbb Z_{\ge 0}.
		\]
		Equivalently, \(B/A\) is a non-negative integer square.
	\end{proposition}
	
	\begin{proof}
		The equation \(Ay^2=B\) is equivalent to
		\[
		y^2=\frac{B}{A}.
		\]
		Thus it has a solution \(y\in\mathbb Q^\times\) if and only if
		\(B/A\) is a square in \(\mathbb Q^\times\). Since
		\[
		[B]=\prod_{i=1}^r [U_{n_i}]
		\qquad\text{in }\mathbb Q^\times/(\mathbb Q^\times)^2,
		\]
		this is equivalent to
		\[
		\prod_{i=1}^r [U_{n_i}]=[A].
		\]
		
		For integral solutions, we need in addition that the rational square \(B/A\) be an integer square. Equivalently, for every rational prime
		\(p\),
		\[
		v_p(B)-v_p(A)
		=
		\sum_{i=1}^r v_p(U_{n_i})-v_p(A)
		\in 2\mathbb Z_{\ge0}.
		\]
		This proves the claimed criterion.
	\end{proof}
	
	\begin{corollary}\label{cor:valuation}
		Let \(A\in \mathbb Z\setminus\{0\}\), and let
		\(n_1,\ldots,n_r\) be pairwise coprime positive integers. Suppose that
		\[
		Ay^2=\prod_{i=1}^r U_{n_i}(P,Q)
		\]
		has a solution \(y\in\mathbb Z\). Then:
		\begin{enumerate}
			\item if \(p\nmid A\), then
			\[
			v_p(U_{n_i})\equiv 0 \pmod 2
			\qquad (1\le i\le r);
			\]
			\item if \(p\mid A\), then \(p\) divides exactly one of the terms \(U_{n_i}\), say
			\(U_{n_j}\), and
			\[
			v_p(U_{n_j})\equiv v_p(A)\pmod 2,
			\qquad
			v_p(U_{n_j})\ge v_p(A).
			\]
		\end{enumerate}
	\end{corollary}
	
	\begin{proof}
		By Lemma~\ref{lem:pairwise-coprime-terms}, the integers
		\[
		U_{n_1},\ldots,U_{n_r}
		\]
		are pairwise coprime. Hence a fixed rational prime \(p\) can divide at most one of them. Taking \(p\)-adic valuations in the equation gives
		\[
		\sum_{i=1}^r v_p(U_{n_i})-v_p(A)=2v_p(y)\in 2\mathbb Z_{\ge 0}.
		\]
		
		If \(p\nmid A\), then \(v_p(A)=0\). Since \(p\) divides at most one factor \(U_{n_i}\), each individual valuation \(v_p(U_{n_i})\) must be even.
		
		If \(p\mid A\), then \(v_p(A)>0\). The displayed equality implies
		\[
		\sum_{i=1}^r v_p(U_{n_i})\ge v_p(A),
		\]
		so \(p\) must divide at least one of the terms \(U_{n_i}\). Since the terms are pairwise coprime, it divides exactly one of them, say \(U_{n_j}\). The valuation identity then becomes
		\[
		v_p(U_{n_j})-v_p(A)=2v_p(y),
		\]
		which proves both the congruence and the inequality.
	\end{proof}
	
	\section{The abc input and $S$-square Lucas terms}
	
	In this section, we isolate the only conditional input used in the paper. The product equation will no longer appear.  We prove that, assuming the number-field \(abc\) conjecture over \(K=\mathbb Q(\sqrt{\Delta})\), a Lucas sequence contains only finitely many terms whose squarefree part is supported on a prescribed finite set of rational primes. In later sections this result is applied after the unconditional valuation-separation argument has forced precisely such an \(S\)-supported condition.
	
	Let
	\[
	K=\mathbb Q(\sqrt\Delta).
	\]
	Under our standing assumptions, \(\Delta\) is not a square in \(\mathbb Q\). Indeed, if
	\(\sqrt\Delta\in\mathbb Q\), then \(\alpha\) and \(\beta\) would be rational algebraic
	integers. Since
	\[
	\alpha\beta=-Q=\pm1,
	\]
	they would be units in \(\mathbb Z\), and hence \(\alpha,\beta\in\{\pm1\}\). This would make \(\alpha/\beta\) a root of unity, contrary to nondegeneracy. Thus \(K\) is a real quadratic field.
	
	All absolute values used below are normalised so as to satisfy the product formula.
	For a nonzero prime ideal \(\mathfrak p\subset \mathcal O_K\), write
	\[
	N\mathfrak p=|\mathcal O_K/\mathfrak p|.
	\]
	
	For a nonzero triple $(a,b,c)\in K^3$ satisfying $a+b=c$, define
	\[
	\rad_K(a:b:c)
	=\frac1{[K:\mathbb Q]}
	\sum_{\mathfrak p\in I(a,b,c)}\log N\mathfrak p,
	\]
	where $I(a,b,c)$ is the set of nonzero prime ideals $\mathfrak p$ for which
	$v_{\mathfrak p}(a),v_{\mathfrak p}(b),v_{\mathfrak p}(c)$ are not all equal.
	We use the absolute logarithmic projective height
	\[
	h(a:b:c)=\frac1{[K:\mathbb Q]}
	\sum_v \log\max\{|a|_v,|b|_v,|c|_v\},
	\]
	where the sum is over all places of $K$. This is the standard projective-height formulation used in Diophantine geometry;
	see, for example, Lang~\cite{Lang1983}.
	
	\begin{conjecture}[abc over $K$]\label{conj:abc}
		For every $\varepsilon>0$ there exists a constant $C_{K,\varepsilon}$ such that for all nonzero triples $(a,b,c)\in K^3$ satisfying $a+b=c$, one has
		\[
		h(a:b:c)
		\le (1+\varepsilon)\rad_K(a:b:c)+C_{K,\varepsilon}.
		\]
	\end{conjecture}
	
	\begin{lemma}\label{lem:units}
		The elements $\alpha$ and $\beta$ are units in $\mathcal O_K$.
	\end{lemma}
	
	\begin{proof}
		They are roots of the monic polynomial $x^2-Px-Q\in\mathbb Z[x]$, hence are
		algebraic integers.  Moreover
		\[
		N_{K/\mathbb Q}(\alpha)=\alpha\beta=-Q=\pm1,
		\qquad
		N_{K/\mathbb Q}(\beta)=\alpha\beta=-Q=\pm1.
		\]
		Thus both are units.
	\end{proof}
	
	\begin{lemma}\label{lem:S-square-growth}
		Let $S$ be a finite set of rational primes.  Suppose that
		\[
		U_n=e s^2,
		\]
		where $e$ is signed squarefree and $\operatorname{Supp}(e)\subseteq S$.  Then
		\[
		\log |s|=\frac n2\log|\alpha|+O_{P,Q,S}(1).
		\]
	\end{lemma}
	
	\begin{proof}
		Since $e$ is signed squarefree and supported on the fixed finite set $S$, its absolute value belongs to a finite set depending only on $S$. Hence
		\[
		\log|U_n|=\log|e|+2\log|s|=2\log|s|+O_S(1).
		\]
		The result follows from Lemma \ref{lem:growth}.
	\end{proof}
	
	\begin{lemma}\label{lem:height-lower}
		There exists a constant $C_1=C_1(P,Q)$ such that, for every representation
		$U_n=e s^2$ with $e\ne0$, one has
		\[
		h(\beta^n:\sqrt\Delta\,e s^2:\alpha^n)
		\ge n\log|\alpha|-C_1.
		\]
		In fact, with the above normalisation, one may take $C_1=0$.
	\end{lemma}
	
	\begin{proof}
		Let
		\[
		a=\beta^n,
		\qquad
		b=\sqrt\Delta\,e s^2,
		\qquad
		c=\alpha^n.
		\]
		The two real embeddings of $K$ interchange $\alpha$ and $\beta$.
		At one embedding, $|c|=|\alpha|^n$; at the conjugate embedding, $|a|=|\alpha|^n$.
		Thus each archimedean contribution is at least $n\log|\alpha|$.
		Dividing by $[K:\mathbb Q]=2$ gives the claim, since the nonarchimedean contributions to the projective height are nonnegative.
	\end{proof}
	
	\begin{lemma}\label{lem:radical-upper-S}
		Let $S$ be a finite set of rational primes.  There exists a constant
		$C_2=C_2(P,Q,S)$ such that, whenever
		\[
		U_n=e s^2,
		\qquad
		e\text{ signed squarefree},
		\qquad
		\operatorname{Supp}(e)\subseteq S,
		\]
		one has
		\[
		\rad_K(\beta^n:\sqrt\Delta\,e s^2:\alpha^n)
		\le \log|s|+C_2.
		\]
	\end{lemma}
	
	\begin{proof}
		Put
		\[
		a=\beta^n,\qquad b=\sqrt\Delta\, e s^2,\qquad c=\alpha^n.
		\]
		By Lemma~\ref{lem:units}, \(a=\beta^n\) and \(c=\alpha^n\) are units of \(\mathcal O_K\).
		
		Since
		\[
		(b)=(\sqrt{\Delta}e)(s)^2,
		\]
		the prime ideals dividing \((\sqrt{\Delta}e)\) lie above the fixed finite set of rational primes \(S\cup \operatorname{Supp}(\Delta)\).
		Their total contribution is bounded by a constant depending only on \(P,Q\), and \(S\). The remaining contributing prime ideals form a subset of the prime ideals dividing \((s)\), up to the fixed set already accounted for.
		
		Since \(s\in\mathbb Z\),
		\[
		N_{K/\mathbb Q}((s))=|s|^{[K:\mathbb Q]}.
		\]
		Therefore
		\[
		\frac1{[K:\mathbb Q]}\sum_{\mathfrak p\mid (s)}\log N\mathfrak p
		\le
		\frac1{[K:\mathbb Q]}\log N_{K/\mathbb Q}((s))
		=
		\log |s|.
		\]
		Combining the bounded contribution from \((\sqrt{\Delta}e)\) with this estimate gives
		\[
		\operatorname{rad}_K(a:b:c)\le \log |s|+C_2(P,Q,S).
		\]
	\end{proof}
	
	\begin{theorem}[Theorem~\ref{thm:conditional}]\label{thm:S-square-finiteness}
		Let \(P,Q\in\mathbb Z\) with \(Q=\pm1\), and suppose that \(\Delta=P^2+4Q>0\) and that \(U_n(P,Q)\) is nondegenerate.
		Assume the number-field abc conjecture over \(K=\mathbb Q(\sqrt\Delta)\).
		Then, for every finite set \(S\) of rational primes, the set of \(n\ge 1\) such that
		\[
		[U_n(P,Q)]\in G_S
		\]
		is finite.
	\end{theorem}
	
	\begin{proof}
		The condition $[U_n]\in\mathcal G_S$ is equivalent to a representation
		\[
		U_n=e s^2,
		\]
		where $e$ is signed squarefree and supported on $S$.  Binet's formula gives
		\[
		\alpha^n-\beta^n=\sqrt\Delta\,U_n
		=\sqrt\Delta\,e s^2,
		\]
		or
		\[
		\beta^n+\sqrt\Delta\,e s^2=\alpha^n.
		\]
		Apply Conjecture \ref{conj:abc} to the triple
		\[
		(\beta^n,\sqrt\Delta\,e s^2,\alpha^n).
		\]
		Fix $0<\varepsilon<1$.  Then
		\[
		h(\beta^n:\sqrt\Delta\,e s^2:\alpha^n)
		\le
		(1+\varepsilon)
		\rad_K(\beta^n:\sqrt\Delta\,e s^2:\alpha^n)
		+C_{K,\varepsilon}.
		\]
		By Lemmas~\ref{lem:height-lower} and~\ref{lem:radical-upper-S}, abc gives
		\[
		n\log|\alpha|
		\le
		(1+\varepsilon)(\log|s|+C_2)+C_{K,\varepsilon}.
		\]
		By Lemma~\ref{lem:S-square-growth},
		\[
		\log|s|=\frac n2\log|\alpha|+O_{P,Q,S}(1).
		\]
		Hence
		\[
		n\log|\alpha|
		\le
		\frac{1+\varepsilon}{2}n\log|\alpha|
		+O_{P,Q,S,K,\varepsilon}(1).
		\]
		Equivalently,
		\[
		\frac{1-\varepsilon}{2}\,n\log|\alpha|
		\le
		O_{P,Q,S,K,\varepsilon}(1).
		\]
		Choosing \(0<\varepsilon<1\), and using \(\log|\alpha|>0\), we conclude that \(n\) is bounded in terms of \(P,Q,S,K\) and \(\varepsilon\). This proves the theorem.
	\end{proof}
	
	This proves Theorem~\ref{thm:conditional}.
	
	\begin{remark}
		Theorem \ref{thm:S-square-finiteness} is independent of any product equation.
		It says that, conditional on abc, a nondegenerate Lucas sequence with $Q=\pm1$ contains only finitely many terms whose signed squarefree parts are supported on a prescribed finite set.
		The coprime product equations studied here provide a natural mechanism forcing precisely this $S$-square condition.
	\end{remark}
	
	\section{Finite reduction of coprime product equations}
	
	We now assemble the two ingredients of the paper.
	The valuation-separation principle forces every factor in a coprime product solution to lie in a fixed square-class set \(G_A\).  The \(abc\)-conditional finiteness theorem then shows that only finitely many nontrivial indices can satisfy this condition.  Hence the original Diophantine problem becomes a finite compatibility problem.
	
	We shall also use the following standard terminology. A rational prime \(p\) is called a primitive prime divisor of \(U_n\) if
	\[
	p \mid U_n \quad \text{and} \quad p\nmid U_m \ (1\le m<n).
	\]
	For any rational prime \(p\) dividing at least one term of the sequence, define
	\[
	z(p):=\min\{n\ge1:p\mid U_n\}.
	\]
	If \(p\) is a primitive prime divisor of \(U_n\), then necessarily \(z(p)=n\).
	Since \(Q=\pm1\), no rational prime divides \(Q\), and the usual rank-of-apparition formalism for Lucas sequences applies to the primes considered below.
	
	\begin{theorem}\label{thm:multi-bound}
		Assume Conjecture \ref{conj:abc}.  Fix $A\in\mathbb Z\setminus\{0\}$ and
		$r\ge1$.  There exists a constant $C=C(P,Q,A)$ such that, if
		$n_1, \ldots,n_r$ are pairwise coprime positive integers and
		\[
		A y^2=\prod_{i=1}^r U_{n_i}(P,Q)
		\]
		for some $y\in\mathbb Z$, then
		\[
		\max_{1\le i\le r} n_i<C.
		\]
	\end{theorem}
	
	\begin{proof}
		By Theorem~\ref{thm:multi-rigidity}, each term $U_{n_i}$ has signed squarefree part supported  $\operatorname{Supp}(A)$.  Apply Theorem~\ref{thm:S-square-finiteness} with $S=\operatorname{Supp}(A)$.
		This gives a uniform bound for every $n_i$.
	\end{proof}
	
	\begin{corollary}\label{cor:two-term-bound}
		Assume Conjecture~\ref{conj:abc}. Then there exists a constant $C(P,Q)$ such that, whenever
		\[
		\Delta y^2=U_m(P,Q)U_n(P,Q),
		\qquad
		\gcd(m,n)=1,
		\]
		one has
		\[
		\max\{m,n\}<C(P,Q).
		\]
		Consequently, only finitely many coprime ordered pairs $(m,n)$ can occur.
	\end{corollary}
	
	\begin{proof}
		This is Theorem~\ref{thm:multi-bound} with $A=\Delta$ and $r=2$.
	\end{proof}
	
	\begin{corollary}\label{cor:square-class-graph-bound}
		Assume Conjecture~\ref{conj:abc}.  In the two-term equation
		\[
		\Delta y^2=U_mU_n,
		\qquad
		\gcd(m,n)=1,
		\]
		all possible solutions arise from finitely many index pairs satisfying the finite square-class compatibility condition
		\[
		[U_m],[U_n]\in\mathcal G_\Delta,
		\qquad
		[U_m][U_n]=[\Delta],
		\]
		together with the integrality condition
		\[
		U_mU_n/\Delta\in\mathbb Z^2.
		\]
	\end{corollary}
	
	\begin{proof}
		The square-class compatibility follows from Corollary~\ref{cor:finite-graph}, and the finiteness of the possible indices follows from Corollary~\ref{cor:two-term-bound}.
		The final integrality condition is Proposition~\ref{prop:two-term-criterion}.
	\end{proof}
	
	\begin{theorem}\label{thm:tpqa}
		Assume Conjecture~\ref{conj:abc} over \(K=\mathbb Q(\sqrt\Delta)\).
		Fix \(A\in \mathbb Z\setminus\{0\}\). Then the set
		\[
		T(P,Q,A):=\{n\geq 2 : [U_n(P,Q)]\in G_A\}
		\]
		is finite.
		
		Moreover, let \(r\geq 1\), and suppose that \(n_1,\ldots,n_r\) are pairwise coprime positive integers satisfying
		\[
		Ay^2=\prod_{i=1}^r U_{n_i}(P,Q)
		\]
		for some \(y\in\mathbb Z\).
		After deleting the trivial indices \(n_i=1\), every remaining index belongs to \(T(P,Q,A)\).
		Consequently, the set of possible nontrivial index sets is finite.
		For each fixed ordering convention, only finitely many ordered tuples occur.
	\end{theorem}
	
	\begin{proof}
		By Theorem~\ref{thm:multi-rigidity}, each occurring term \(U_{n_i}\) has signed squarefree part supported on \(\operatorname{Supp}(A)\).
		Hence, whenever \(n_i\geq 2\), one has
		\[
		[U_{n_i}(P,Q)]\in G_A,
		\]
		and therefore \(n_i\in T(P,Q,A)\). The set \(T(P,Q,A)\) is finite by Theorem~\ref{thm:S-square-finiteness}, applied with \(S=\operatorname{Supp}(A)\).
		
		Thus every nontrivial index appearing in a solution belongs to the fixed finite set \(T(P,Q,A)\). Since the indices are required to be pairwise coprime, the possible nontrivial index sets are among the subsets of this finite set satisfying the pairwise coprimality condition. There are only finitely many such subsets. Once an ordering convention is fixed, this also gives only finitely many ordered tuples.
	\end{proof}

\begin{corollary}\label{cor:coprime-graph}
Assume Conjecture~\ref{conj:abc} and fix \(A\in\mathbb Z\setminus\{0\}\).
After deleting trivial factors \(U_1=1\), the nontrivial indices occurring in any solution of
\[
Ay^2=\prod_{i=1}^r U_{n_i}(P,Q),
\qquad \gcd(n_i,n_j)=1 \ (i\neq j),
\]
form a clique in the coprimality graph on the finite set
\[
T(P,Q,A)=\{n\geq 2 : [U_n(P,Q)]\in G_A\}.
\]
The remaining conditions are the finite square-class compatibility condition
\[
\prod_i [U_{n_i}]=[A]
\quad\text{in } \mathbb Q^\times/(\mathbb Q^\times)^2
\]
and the integrality condition from Proposition~\ref{prop:product-square}.
\end{corollary}

\begin{proof}
By Theorem~\ref{thm:tpqa}, every nontrivial index in such a solution belongs to \(T(P,Q,A)\). The pairwise coprimality assumption says precisely that these indices form a clique in the coprimality graph on \(T(P,Q,A)\).
The square-class compatibility and the integrality condition are those of Proposition~\ref{prop:product-square}.
\end{proof}

Together with Theorem~\ref{thm:tpqa}, this proves Corollary~\ref{thm:finite-reduction}.

We shall use the primitive divisor theorem for Lucas sequences in the standard form applicable to nondegenerate Lucas pairs. Our convention uses the recurrence \(U_{n+2}=PU_{n+1}+QU_n\), so that the associated characteristic polynomial is \(x^2-Px-Q\). This differs from the classical \(U_n(P,Q)\) notation only by the harmless replacement of the second parameter by \(-Q\).

\begin{corollary}\label{cor:prim-divisor}
	Let
	\[
	Z_A=\{z(p):p\mid A\text{ and }z(p)\text{ exists}\}.
	\]
	There exists a constant \(N_0=N_0(P,Q)\) such that the following holds.
	Suppose that \(n>N_0\), \(n\notin Z_A\), and \(n\) occurs as one of the indices in a coprime product solution
	\[
	Ay^2=\prod_{i=1}^r U_{n_i}(P,Q),\qquad (n_i,n_j)=1\quad (i\ne j).
	\]
	Then \(U_n\) has a primitive prime divisor \(p\nmid A\), and every primitive prime divisor \(p\nmid A\) of \(U_n\) satisfies
	\[
	v_p(U_n)\ge 2.
	\]
	Thus large admissible indices force primitive divisors outside \(A\) to occur with multiplicity at least two.
\end{corollary}

\begin{proof}
	By the primitive divisor theorem of Bilu--Hanrot--Voutier~\cite{BiluHanrotVoutier2001}, applied to the corresponding nondegenerate Lucas pair, there exists a constant \(N_0=N_0(P,Q)\) such that \(U_n\) has a primitive prime divisor for every \(n>N_0\). Enlarging \(N_0\), if necessary, absorbs the finitely many exceptional indices allowed by the theorem. Let \(p\) be a primitive prime divisor of \(U_n\).
	
	If \(p\mid A\), then by primitivity the rank of apparition of \(p\) is \(z(p)=n\).
	This would imply \(n\in Z_A\), contrary to the hypothesis. Hence \(p\nmid A\).
	
	Since \(n\) occurs as one of the indices in a coprime product solution, Lemma~\ref{lem:valuation-parity-multi} applies to the corresponding factor \(U_n\). Hence every prime \(q\nmid A\) dividing \(U_n\) satisfies
	\[
	v_q(U_n)\equiv 0 \pmod 2.
	\]
	In particular, this applies to every primitive prime divisor \(p\nmid A\) of \(U_n\).
	Since \(p\mid U_n\), we have \(v_p(U_n)\ge 1\), and the preceding congruence therefore gives
	\[
	v_p(U_n)\ge 2.
	\]
\end{proof}

\begin{remark}
	The terminology is meant in the weak valuation-theoretic sense that a primitive divisor is forced to occur to order at least two. This is analogous in spirit to Wieferich-type phenomena studied via the \(abc\) conjecture, for instance by Silverman~\cite{Silverman1988}.
\end{remark}

\section{Power-class versions of valuation separation}

The preceding arguments are not intrinsically restricted to squares. The coprimality of the indices separates \(p\)-adic valuations term by term, and therefore the same mechanism applies to \(k\)-th powers.

Let \(k\ge 2\).  We shall say that a nonzero integer \(e\) is \(k\)-th-power-free if \(p^k\nmid e\) for every rational prime \(p\).
Equivalently, in the prime factorisation of \(|e|\), every exponent lies between \(0\) and \(k-1\).

For a nonzero integer \(N\), we use the standard decomposition
\[
N=e s^k,
\]
where \(s\in\mathbb Z_{\ge 1}\), \(e\in\mathbb Z\) is \(k\)-th-power-free, and \(e\) carries the sign of \(N\). With this convention the decomposition is unique.

\begin{theorem}[Power-class rigidity]\label{thm:power-class-rigidity}
	Let \(k\ge 2\), let \(A\in \mathbb Z\setminus\{0\}\), and suppose that \(n_1,\ldots,n_r\) are pairwise coprime positive integers.  If
	\[
	A y^k=\prod_{i=1}^r U_{n_i}(P,Q)
	\]
	for some \(y\in\mathbb Z\), then for every \(i\) there exist \(s_i\in \mathbb Z_{\ge 1}\) and a \(k\)-th-power-free integer \(e_i\) such that
	\[
	U_{n_i}=e_i s_i^k,
	\]
	and every prime divisor of \(e_i\) divides \(A\). Equivalently, for every rational prime \(p\nmid A\), one has
	\[
	v_p(U_{n_i})\equiv 0 \pmod k
	\qquad (1\le i\le r).
	\]
\end{theorem}

\begin{proof}
	Fix a rational prime \(p\nmid A\). Taking \(p\)-adic valuations in
	\[
	A y^k=\prod_{i=1}^r U_{n_i}
	\]
	gives
	\[
	\sum_{i=1}^r v_p(U_{n_i})=k v_p(y).
	\]
	The right-hand side is divisible by \(k\). By Lemma~\ref{lem:pairwise-coprime-terms}, the integers \(U_{n_1},\ldots,U_{n_r}\) are pairwise coprime. Hence \(p\) divides at most one of them. Therefore each individual valuation \(v_p(U_{n_i})\) is divisible by \(k\).
	
	Writing each \(U_{n_i}\) in the above standard form
	\[
	U_{n_i}=e_i s_i^k
	\]
	with \(e_i\) \(k\)-th-power-free, a prime \(p\) divides \(e_i\) if and only if \(v_p(U_{n_i})\not\equiv 0\pmod k\). The previous paragraph shows that this can happen only when \(p\mid A\).
\end{proof}

This is the \(k\)-th power analogue of Theorem~\ref{thm:multi-rigidity}.

The corresponding \(abc\)-conditional finiteness statement also extends without essential change.

\begin{theorem}\label{thm:s-k-finiteness}
	Let \(k\ge 2\), and assume the number-field \(abc\) conjecture over \(K=\mathbb Q(\sqrt\Delta)\).  Let \(S\) be a finite set of rational primes. Then the set of positive integers \(n\) for which
	\[
	U_n(P,Q)=e s^k
	\]
	with \(e\) \(k\)-th-power-free and \(\operatorname{Supp}(e)\subseteq S\) is finite.
\end{theorem}

\begin{proof}
	Suppose that
	\[
	U_n=e s^k,
	\]
	where \(e\) is \(k\)-th-power-free and supported on \(S\). Since \(S\) and \(k\) are fixed, the integer \(e\) belongs to a finite set.  By Lemma~\ref{lem:growth},
	\[
	\log |s|=\frac{n}{k}\log|\alpha|+O_{P,Q,S,k}(1).
	\]
	
	As before, Binet's formula gives
	\[
	\beta^n+\sqrt\Delta\, e s^k=\alpha^n.
	\]
	Apply Conjecture~\ref{conj:abc} to the triple
	\[
	(\beta^n,\sqrt\Delta\, e s^k,\alpha^n).
	\]
	The proof of Lemma~\ref{lem:height-lower}, which uses only the archimedean sizes of \(\alpha^n\) and \(\beta^n\), gives
	\[
	h(\beta^n:\sqrt\Delta\, es^k:\alpha^n)
	\ge n\log|\alpha|-O_{P,Q}(1).
	\]
	Similarly, the proof of Lemma~\ref{lem:radical-upper-S} gives
	\[
	\operatorname{rad}_K(\beta^n:\sqrt\Delta\, es^k:\alpha^n)
	\le \log |s|+O_{P,Q,S,k}(1).
	\]
	Indeed, \(\alpha\) and \(\beta\) are units, so the finite primes contributing to the radical lie above primes dividing \(\sqrt\Delta e\) or \(s\); the contribution from \(\sqrt\Delta e\) is bounded in terms of \(P,Q,S,k\), while the contribution from \((s)\) is at most \(\log |s|\).
	Therefore, for every \(\varepsilon>0\), the \(abc\) conjecture gives
	\[
	n\log|\alpha|
	\le
	(1+\varepsilon)\log |s|+O_{P,Q,S,K,k,\varepsilon}(1).
	\]
	Using the estimate for \(\log |s|\), we obtain
	\[
	n\log|\alpha|
	\le
	\frac{1+\varepsilon}{k}\, n\log|\alpha|
	+O_{P,Q,S,K,k,\varepsilon}(1).
	\]
	Choose \(0<\varepsilon<k-1\). Then
	\[
	1-\frac{1+\varepsilon}{k}>0,
	\]
	and since \(\log|\alpha|>0\), the last inequality bounds \(n\). This proves the finiteness claim.
\end{proof}

The valuation-separation theorem reduces the \(k\)-th power case to Lucas terms which are \(k\)-th powers up to primes in a fixed finite set. The corresponding \(abc\)-conditional finiteness statement is as follows.

\begin{corollary}\label{cor:coprime-tuples}
	Let \(k\ge2\), and assume the number-field \(abc\) conjecture over \(K=\mathbb Q(\sqrt\Delta)\).  Fix \(A\in\mathbb Z\setminus\{0\}\). Then there exists a constant \(C=C(P,Q,A,k)\) such that, whenever \(n_1,\ldots,n_r\) are pairwise coprime positive integers and
	\[
	A y^k=\prod_{i=1}^r U_{n_i}(P,Q)
	\]
	for some \(y\in\mathbb Z\), one has
	\[
	\max_{1\le i\le r} n_i<C.
	\]
	In particular, for fixed \(A\) and \(k\), only finitely many nontrivial pairwise coprime index tuples can occur, after deleting factors \(U_1=1\).
\end{corollary}

\begin{proof}
	By the power-class rigidity theorem, each occurring term \(U_{n_i}\) can be written in the form
	\[
	U_{n_i}=e_i s_i^k,
	\]
	where \(e_i\) is \(k\)-th-power-free and supported on \(\operatorname{Supp}(A)\).
	The \(S\)-supported \(k\)-th-power finiteness theorem, applied with \(S=\operatorname{Supp}(A)\), gives a uniform bound for all such \(n_i\).
\end{proof}

\begin{remark}
	For \(k=2\), this recovers the square-class rigidity and the \(abc\)-conditional boundedness results proved above.  The point of the higher-power formulation is that the argument is governed by the separation of valuations among coprime Lucas factors, rather than by any special feature of quadratic square classes.
\end{remark}
	
	\section{Examples}
	
	\begin{example}[Fibonacci sequence]
		Let \((P,Q)=(1,1)\), so that \(U_n=F_n\) and
		\(\Delta=5\).  Consider
		\[
		5y^2=F_mF_n,\qquad (m,n)=1.
		\]
		Since
		\[
		G_\Delta=\langle -1,5\rangle
		\subset \mathbb Q^\times/(\mathbb Q^\times)^2,
		\]
		Corollary~\ref{cor:two-term-rigidity} gives
		\[
		F_m=e_m s_m^2,\qquad F_n=e_n s_n^2,
		\qquad e_m,e_n\in\{\pm1,\pm5\}.
		\]
		As \(F_m,F_n>0\), this reduces to
		\[
		e_m,e_n\in\{1,5\}.
		\]
		The square-class condition
		\[
		[F_m][F_n]=[5]
		\]
		therefore says that exactly one of \(F_m,F_n\) is a square and the other is five times a square.
		
		The first few values illustrate the obstruction:
		\[
		\begin{array}{c|c|c}
			n & F_n & \operatorname{sqf}(F_n)\\
			\hline
			1 & 1 & 1\\
			2 & 1 & 1\\
			3 & 2 & 2\\
			4 & 3 & 3\\
			5 & 5 & 5\\
			6 & 8 & 2\\
			12 & 144 & 1
		\end{array}
		\]
		Thus the indices \(3,4,6\), for instance, cannot occur as one of the two coprime factors in a solution of \(5y^2=F_mF_n\), since their squarefree parts contain primes other than \(5\).
		
		On the other hand,
		\[
		F_{12}=144=12^2,\qquad F_5=5,
		\]
		and \((12,5)=1\).  Hence
		\[
		5\cdot 12^2=F_{12}F_5,
		\]
		so \((m,n,y)=(12,5,12)\) is a solution. The reversed pair \((5,12,12)\) gives the same unordered solution.
		
		By Cohn's theorem \cite{Cohn1964}, the only nonzero square Fibonacci numbers are
		\[
		F_1=F_2=1,\qquad F_{12}=144.
		\]
		Consequently, in any nontrivial solution of the above equation, the square-class criterion forces the square factor to come from this very small list, while the other factor must be five times a square.
		
	This illustrates that, under the coprimality hypothesis, the square-class restriction is imposed on each factor separately rather than only on the product.
	\end{example}

	\begin{example}[Pell-type Lucas sequence]
		Let $(P,Q)=(2,1)$.  Then
		\[
		\Delta=P^2+4Q=8,
		\qquad
		\rad(\Delta)=2,
		\]
		and the sequence begins
		\[
		U_1=1,
		\quad U_2=2,
		\quad U_3=5,
		\quad U_4=12,
		\quad U_5=29.
		\]
		If
		\[
		8y^2=U_mU_n,
		\qquad
		\gcd(m,n)=1,
		\]
		then the signed squarefree parts of both $U_m$ and $U_n$ must be supported only at $2$. Hence an index $n$ with $U_n=5$ or $U_n=29$ cannot occur with a coprime partner in such an equation.
	\end{example}
	
	\begin{example}[Multi-product obstruction]
		For the Fibonacci sequence, let $A=5$ and suppose that
		\[
		5y^2=F_{n_1}F_{n_2}\cdots F_{n_r}
		\]
		with $n_1,
		\ldots,n_r$ pairwise coprime. Then each individual Fibonacci number $F_{n_i}$ must have signed squarefree part supported on $\{5\}$.  Thus the obstruction is termwise, not merely global: a single factor $F_{n_i}$ containing a prime $p\ne5$ to odd valuation already rules out the equation, regardless of the other coprime factors.
	\end{example}
	
	\begin{remark}
		The theorem should not be tested by looking at a single composite-index term in isolation. A composite-index term $U_n$ may have a large squarefree part in general. The result says that such a term cannot appear inside a coprime product equation of the form \eqref{eq:product} unless its squarefree part is supported on the fixed set of primes dividing $A$.
	\end{remark}

\end{document}